\documentclass[12pt]{amsart}

\usepackage{amsmath,amssymb}
\usepackage{amsfonts}
\usepackage{amsthm}
\usepackage{latexsym}
\usepackage{graphicx}
\usepackage{epsfig}

\newtheorem{theorem}{Theorem}[section]
\newtheorem{lemma}[theorem]{Lemma}

\newtheorem{corollary}[theorem]{Corollary}

\newtheorem{remark}{Remark}[section]
\newtheorem{definitions}{Definition}[section]

\makeatletter \@addtoreset{equation}{section} \makeatother

\def\ddt{\frac{d}{dt}}
\def\ppt{\frac{\partial}{\partial t}}
\def\PP{{\mathrm P}}
\def\QQ{{\mathrm Q}}
\def\RR{{\mathrm R}}
\def\WW{{\mathrm W}}

\def\EE{{\mathrm E}}

\def\Rc{{\mathrm {Rc}}}

\def\Rm{{\mathrm {Rm}}}

\begin{document}

\title{Curvature Pinching Estimate And Singularities Of The Ricci Flow}

\author{Xiaodong Cao$^*$
}
\thanks{$^*$Research partially supported by NSF grant DMS 0904432 and by the Jeffrey Sean Lehman Fund from Cornell University}

\address{Department of Mathematics,
 Cornell University, Ithaca, NY 14853-4201}
\email{cao@math.cornell.edu}



\renewcommand{\subjclassname}{%
  \textup{2000} Mathematics Subject Classification}

\date{Oct. 28, 2010}

\maketitle

\markboth{Xiaodong Cao} {Curvature Pinching Estimate and Singularities of The Ricci Flow}

\begin{abstract}In this paper, we first derive a pinching estimate on the traceless Ricci curvature in term of scalar curvature and Weyl tensor under the Ricci flow. Then we apply this estimate to study finite-time singularity behavior. We show that if the scalar curvature is uniformly bounded, then the Weyl tensor has to blow up, as a consequence, the corresponding singularity model must be Ricci flat with non-vanishing Weyl tensor.
\end{abstract}

\section{\textbf{Introduction}}

Let $(M,g)$ be a smooth, closed $n$-dimensional Riemannian manifold. In his seminal paper \cite{H3}, R. Hamilton proved that any closed $3$-manifold which admits a Riemannian metric with strictly positive Ricci curvature must also admits a metric of constant positive sectional curvature. He showed that the original metric can be deformed into the constant-curvature metric by introducing the Ricci flow:
\begin{align}
  \frac{\partial }{\partial t}g_{ij}= -2R_{ij}. \label{rf}
\end{align}
The Ricci flow equation is a (weakly) parabolic partial differential equation system. Its short time existence was first proved by Hamilton (\cite{H3}) and later the proof was simplified by D. DeTurck (\cite{deturk1}). One of the main subjects in the study of Ricci flow is the understanding of long time behavior and formation of singularities. More precisely, we would like to ask when the flow can exist for all time; and if the flow only exits up to finite time, we would like to understand the profile of finite-time singularities, which in general will permit us to understand geometry and topology of the underlying manifold better.\\

A solution $(M,g(t))$ to the Ricci flow equation (\ref{rf}) is a (finite-time) maximal
solution if it is defined for  $t\in [0, T),~ T<\infty$. In
\cite{H3}, Hamilton proved that the whole Riemannian curvature
tensor $\Rm$ blows up as $t\rightarrow T$, i.e., $\lim \sup_{[0,T)} |\Rm|=\infty$. In \cite{sesum051}, N.
Sesum showed that in fact the Ricci curvature tensor $\Rc$ blows
up as $t\rightarrow T$, i.e., $\lim \sup_{[0,T)} |\Rc|=\infty$.
In other words, if the norm of Riemannian curvature or Ricci curvature
is uniformly bounded on $[0, T)$, then the flow can be smoothly extended
past $T$.  In \cite{wangb08}, B. Wang extended the above results even
further by showing that that if the Ricci curvature tensor $\Rc$ is
uniformly bounded from below and moreover, the space-time integral of
scalar curvature $\RR$ is bounded, namely,
$$\int_0^T \int_M |\RR|^{\alpha} \leq C,~\alpha \geq \frac{n+2}{2},$$
then the Ricci flow can be smoothly extended past $T$. Similar
type results also appeared in \cite{ye084}   by R. Ye and in
\cite{mc10} by L. Ma and L. Cheng.\\

There is a well-known conjecture that the
scalar curvature $\RR$ should also blow up at the singular time $T$. Recently, J. Enders, R. M\"{u}ller, P. M. Topping (\cite{emt10}), and N. Le, Sesum (\cite{ls10}) partially confirmed this conjecture in the case of Type I maximal solutions. Their methods both using blow-up argument based on Perelman's entropy functionals, reduced distance and pseudolocality theorem. \\

\begin{definitions}
A solution $(M,g(t))$, $0\leq t < T <\infty$, is called a Type I maximal solution of the Ricci flow, if there exists a constant $C<\infty$ such that the curvature satisfies $$|\Rm| \leq \frac{C}{T-t}.$$ Otherwise it's a Type II maximal solution of the Ricci flow.
\end{definitions}

In this paper, we study the blow-up behavior of different components of the curvature tensor under the Ricci flow, and their consequences in dilation limits.

For simplicity, we also use the following convention: the
constants $c_i$ only depend on the dimension $n$, but not on the
initial metric $g(0)$; while the constants $C_i$ depend not only
on the dimension $n$, but also on the initial metric $g(0)$. We
also restrict ourselves to the case of \textbf{positive scalar
curvature}, even though that most estimates in this paper can be
carried to the general case.

The rest of this paper organized as follows. In Section 2, we
briefly review the orthogonal decomposition of Riemannian curvature and evolution of curvatures under the Ricci flow. In Section 3, we derive a pinching estimate on the traceless Ricci curvature tensor. As one application, we obtain some information about curvature blow up at finite-time; the second application is on manifolds with positive isotropic curvature. In Section 4, we discuss singularity models and apply the pinching estimate in Section 3 to study the dilation limit.


\section{\textbf{Decomposition and Evolution of Curvature Tensors}}

In this section, we will first give a brief introduction of
curvature decomposition of the Riemannian manifold $(M\sp n,g)$
and some relations of geometric conditions. Then we will recall
some evolution formulae for various curvature tensors, for more
details, please see \cite{H3}. We use $g_{ij}$ to denote the local
components of  metric $g$ and its inverse by $g^{ij}$. In this
paper we use $\Rm$ to denote the $(4,0)$ Riemannian curvature
tensor instead of the $(3,1)$ Riemannian curvature tensor, we
denote its local components by $\RR_{ijkl}$. Let $\Rc$ be the
Ricci curvature with local components $\RR_{ik}=g^{jl}\RR_{ijkl}$,
and let $\RR=g^{ik}\RR_{ik}$ be the scalar curvature. We first
recall the Kulkarni-Nomizu product for two symmetric tensor $h$,
$k$ is defined as:
\begin{align*}
h\circ k(v_1, v_2, v_3, v_4)=&h(v_1, v_3)k(v_2,v_4)+h(v_2,
v_4)k(v_1, v_3)\\ & -h(v_1, v_4)k(v_2, v_3)- h(v_2, v_3)k(v_1,
v_4).
\end{align*}
The Einstein tensor or traceless Ricci tensor $\EE$ is defined as

$$\EE_{ij}=\RR_{ij}-\frac{\RR}{n} g_{ij}.$$
When $n\geq 4$, we can decompose the $(4,0)$ Riemannian curvature
tensor $\Rm$ in the following way:

$$\Rm=\frac{\RR}{2n(n-1)}g\circ g+\frac{1}{n-2} \EE\circ g +\WW,$$
here $\WW$ is the Weyl curvature tensor. And the above
decompositions are orthogonal.\\

In local coordinates, we can write
\begin{align*}
\WW_{ijkl}=&
\RR_{ijkl}-\frac{1}{n-2}(g_{ik}\RR_{jl}+g_{jl}\RR_{ik}
-g_{il}\RR_{jk}-g_{jk}\RR_{il})\\
&+\frac{1}{(n-1)(n-2)}\RR(g_{ik}g_{jl}-g_{il}g_{jk}).
\end{align*}
It is well-known that under conformal change of the metric
$g^{'}=e^{u}\cdot g$ for some function $u$, then
$\WW^{'}=e^{u}\cdot \WW$. If we view the Weyl tensor as a $(3,1)$
tensor, then $\WW^{'}=\WW$, i.e., the $(3,1)$ Weyl tensor is a
conformal invariant.\\

Under the Ricci flow, the Ricci curvature is evolving by $$\ppt
\Rc=\triangle \Rc+2\Rm(\Rc,\cdot)-2\Rc^2,$$ where
$\Rc^2_{ij}=\RR_{ik}\RR_{kj}$ and the scalar curvature evolves by
\begin{align}\label{scalar}
\ppt \RR=\triangle
\RR+2|\Rc|^2.
\end{align}
As a direct consequence of (\ref{scalar}), in all dimensions, the
positivity (or any lower bound) of the scalar curvature is
preserved by the Ricci
flow. In dimension $3$, the positivity of Ricci curvature is preserved (see \cite{H3}). In dimension at least $4$, positivity of curvature operator is preserved (\cite{HPCO} and \cite{Hsurvey}).\\

In \cite{mm88}, M. Micallef and J. D. Moore introduced a new curvature
condition, positive isotropic curvature. A Riemannian manifold of
dimension at least $4$ is said to have positive isotropic curvature,
if for every orthonormal $4$-frame $\{e_1, e_2, e_3, e_4\}$, we have
$$\RR_{1313}+\RR_{1414}+\RR_{2323}+\RR_{2424}-2\RR_{1234}>0.$$
Using minimal surface technique, they proved that
 any compact, simply connected manifold with
 positive isotropic curvature is homeomorphic to
 $S\sp n$. In the
same paper, they observed that the positivity of isotropic
curvature is implied by several other commonly used curvature
conditions, such as positive curvature operator and pointwise
$\frac14$-pinched condition. In dimension $4$, Hamilton
\cite{HPIC} proved that the positivity of isotropic curvature is
preserved by the Ricci flow. This result has been extended to
higher dimensions by S. Brendle and R. Schoen \cite{bs091} and
also by H. Nguyen \cite{nguyen10} independently. Brendle and
Schoen further proved the differentiable sphere theorem, which has
been a long time conjecture since the (topological)
$\frac14$-pinched sphere theorem proved by M. Berger
\cite{berger60} and W. Klingenberg \cite{kl61} around 1960. More
precisely, Brendle and Schoen showed that any compact Riemannian
manifold with pointwise $\frac14$-pinched
sectional curvature is diffeomorphic to a spherical space form \cite{bs091}.\\

Another interesting geometric operator in Riemannian geometry, the
Weitzenb\"{o}ck operator $\PP$, is defined as
$$\PP=\Rc \circ g -2\Rm= \frac{(n-2)\RR}{n(n-1)}g\circ
g+\frac{n-4}{n-2} \EE\circ g -\WW,$$ or in local coordinates,

$$\PP_{ijkl}=(g_{ik}\RR_{jl}+g_{jl}\RR_{ik}
-g_{il}\RR_{jk}-g_{jk}\RR_{il})-2R_{ijkl}.$$

It is known that in dimension $4$, positive isotropic curvature is
equivalent to positive Weitzenb\"{o}ck operator (see for example,
\cite{mm88, mw93, noronha94, noronha97}). For an even dimensional
Riemannian manifold of $n>4$, positive isotropic curvature implies
positive Weitzenb\"{o}ck operator (\cite[Proposition
1.1]{seaman93}).

\section{\textbf{Curvature Pinching Estimate}}

The general evolution formulae of curvature tensors suggests that
the orthogonal parts of Riemannian curvature tensor is not
evolving totally independently to each other, one part might
depend on the other part(s). An interesting question in the study
of the Ricci flow is which orthogonal part(s) needs to blow up at
a finite-time $T$ when singularity occurs. In other words, if
these parts are uniformly bounded up to time $T$, then the Ricci
flow
can be smoothly extended past $T$.\\

Our main theorem in this section is the following estimate, which
 says that the traceless Ricci part
$|\EE|$ can be controlled by  the scalar curvature $\RR$ and  Weyl
tensor  $|\WW|$. This improves an earlier result of D. Knopf
\cite{knopf092}.

\begin{theorem}
\label{thm:einstein} Let $(M\sp n, g(t))$, $t\in [0,T)$, be a
solution to the Ricci flow on a closed Riemannian manifold of
dimension $n\geq 3$, then there exist constants  $C_1(n,g_0)>0$
and $c_2(n)\geq 0$, such that for all $t\geq 0$, one has $\RR+c>0$
and
\begin{align}\label{gr}
\frac{|\EE|}{\RR+c} \leq C_1+c_2 \max_{M\times [0,t]}
\sqrt{\frac{|\WW|}{\RR+c}}.
\end{align}
Furthermore, if $\RR>0$ at $t=0$, then we have
\begin{align}\label{pr} \frac{|\EE|}{\RR} \leq C_1+c_2 \max_{M\times [0,t]}
\sqrt{\frac{|\WW|}{\RR}}.
\end{align}
\end{theorem}

\begin{remark}
In  \cite{knopf092}, Knopf first proved a rather surprising
result, namely he showed that under the Ricci flow, there exist
constants $c(g_0)\geq 0$, $C_1(n,g_0)>0$ and $c_2(n)>0$ such that
for all $t\geq 0$, one has $R+c>0$ and
\begin{align}\label{dk}
\frac{|\EE|}{\RR+c} \leq C_1+c_2 \max_{s\in[0,t]}
\sqrt{\frac{|\WW|_{\max}(s)}{\RR_{\min}(s)+c}}.
\end{align} In other words, the traceless Ricci part $|\EE|$ can be
controlled by the maximum and minimum of the scalar curvature
$\RR$ and the maximum of the Weyl tensor $|\WW|$. Notice that $|\WW|_{\max}$ and $\RR_{\min}$ may actually be achieved at different space-time points in (\ref{dk}),  while estimates (\ref{gr}) and (\ref{pr}) proved here do not have this problem. This makes (\ref{gr}) and (\ref{pr}) more powerful in studying dilation limits of singularities.
\end{remark}

\begin{remark}
We state the theorem both for the general case and for the
positive scalar curvature case. But for simplicity, we will only
prove (\ref{pr}) here, the proof of (\ref{gr}) is similar.
\end{remark}

\begin{remark}
The estimate (\ref{pr}) is scaling invariant, so it still holds
for normalized Ricci flow and also for Ricci flow solutions exist
for all time $[0, \infty)$. In the special case of
K\"{a}hler-Ricci flow, it is known that the scalar curvature is
bounded (this is claimed by G. Perelman and a detailed proof is
given by Sesum and G. Tian \cite{st08}), hence the whole curvature
tensor blows up if and only if the Weyl tensor $\WW$ blows up
(also see \cite{mc10}).
\end{remark}

For our purpose, we perform a rather general calculation here. For
any positive number $\gamma$, define
$$f=\frac{|\EE|^2}{\RR^{\gamma}}=\frac{|\Rc|^2}{\RR^{\gamma}}-
\frac{1}{n}\RR^{2-\gamma},$$ then $f$ satisfies the following
evolution equation:
\begin{lemma} Under the Ricci flow, we have
\begin{align}
\ppt f= &\triangle f+\frac{2(\gamma -1)}{\RR}\nabla f \cdot \nabla
\RR -\frac{2}{\RR^{2+\gamma}} |\RR\nabla_i \RR_{jk}-\nabla_i \RR
\RR_{jk}|^2 -\frac{(2-\gamma)(\gamma-1)}{\RR^{2}}|\nabla \RR|^2 f
\nonumber
\\ &
-\frac{2(2-\gamma)}{n}{\RR^{1-\gamma}}|\Rc|^2 \nonumber
+\frac{4}{\RR^{\gamma}}\Rm(\Rc,\Rc)-\frac{2\gamma}{\RR^{1+\gamma}}
|\Rc|^4-\frac{(2-\gamma)(\gamma-1)}{n\RR^{\gamma}}
\nonumber \\
=&\triangle f+\frac{2(\gamma -1)}{\RR}\nabla f \cdot \nabla \RR
-\frac{2}{\RR^{2+\gamma}} |\RR\nabla_i \RR_{jk}-\nabla_i \RR
\RR_{jk}|^2
-\frac{(2-\gamma)(\gamma-1)}{\RR^{2}}|\nabla \RR|^2 f \nonumber \\
&+\frac{2}{\RR^{1+\gamma}}[(2-\gamma)|\Rc|^2(|\Rc|^2-\frac{1}{n}\RR^2)
-2(|\Rc|^4-\RR\cdot
\Rm(\Rc,\Rc)]-\frac{(2-\gamma)(\gamma-1)}{n\RR^{\gamma}}.
\nonumber
\end{align}
\end{lemma}

\begin{proof}
We have $$\ppt |\Rc|^2=\triangle |\Rc|^2-2|\nabla
\Rc|^2+4\Rm(\Rc,\Rc),$$ where
$$\Rm(\Rc,\Rc)=\RR_{abcd}\RR_{ac}\RR_{bd}.$$
We can further express the term as
$$\Rm(\Rc,\Rc)=\RR_{abcd}\RR_{ac}\RR_{bd}=\frac{1}{n-2}(\frac{2n-1}{n-1}
|\Rc|^2R-2\Rc^3-\frac{\RR^3}{n-1})+\WW(\Rc,\Rc),$$ hence we arrive
at,
$$\ppt |\Rc|^2=\triangle |\Rc|^2-2|\nabla
\Rc|^2+\frac{4}{n-2}(\frac{2n-1}{n-1}|\Rc|^2R-2\Rc^3-\frac{\RR^3}{n-1})
+\WW(\Rc,\Rc).$$ Using this together with the evolution equation
of the scalar curvature $$\ppt \RR=\triangle \RR+2|\Rc|^2,$$ we
have the following two equations:

\begin{align*}
\ppt (\frac{|\Rc|^2}{\RR^{\gamma}})= &\, \triangle
(\frac{|\Rc|^2}{\RR^{\gamma}})+\frac{2(\gamma -1)}{\RR}\nabla
(\frac{|\Rc|^2}{\RR^{\gamma}}) \cdot \nabla \RR
-\frac{2}{\RR^{2+\gamma}} |\RR \nabla_i \RR_{jk}-\nabla_i \RR
\RR_{jk}|^2\\
&\,-\frac{(2-\gamma)(\gamma-1)}{\RR^{2+\gamma}}|\Rc|^2|\nabla
\RR|^2
+\frac{4}{\RR^{\gamma}}\Rm(\Rc,\Rc)-\frac{2\gamma}{\RR^{1
+\gamma}}|\Rc|^4,
\end{align*}
and
\begin{align*}
\ppt \RR^{2-\gamma}= \triangle \RR^{2-\gamma}+\frac{2(\gamma
-1)}{\RR}\nabla \RR^{2-\gamma} \cdot \nabla \RR
+2(2-\gamma){\RR^{1-\gamma}}|\Rc|^2.
\end{align*}
The lemma then follows.
\end{proof}
We can also rewrite the above lemma in the following way.
\begin{lemma} Under the Ricci flow, we have
\begin{align*}
\ppt f= &\, \triangle f+\frac{2(\gamma -1)}{\RR}\nabla f \cdot
\nabla \RR -\frac{2}{\RR^{2+\gamma}} |\RR\nabla_i
\RR_{jk}-\nabla_i \RR \RR_{jk}|^2
-\frac{(2-\gamma)(\gamma-1)}{\RR^{2}}|\nabla \RR|^2 f
\\
&\,+\frac{2}{\RR^{1+\gamma}}\left[(2-\gamma)|\Rc|^2|\EE|^2
-2\QQ+2\RR
\WW(\Rc,\Rc)\right]-\frac{(2-\gamma)(\gamma-1)}{n\RR^{\gamma}}\\
= &\,\triangle f+\frac{2(\gamma -1)}{\RR}\nabla f \cdot \nabla \RR
-\frac{2}{\RR^{2+\gamma}} |\RR\nabla_i \RR_{jk}-\nabla_i \RR
\RR_{jk}|^2 -\frac{(2-\gamma)(\gamma-1)}{\RR^{2}}|\nabla \RR|^2 f
\\
&\,+\frac{2}{\RR^{1+\gamma}}\left[-\gamma
|\EE|^4+\left(\frac{2(n-2)}{n(n-1)}
-\frac{\gamma}{n}\right)|\RR|^2|\EE|^2
-\frac{4}{n-2}\RR \EE^3+2\RR \WW(\EE,\EE)\right]\\ &\,-\frac{(2-\gamma)(\gamma-1)}{n\RR^{\gamma}},
\end{align*}where
$\QQ=|\Rc|^4-\frac{\RR}{n-2}(\frac{2n-1}{n-1}\RR|\Rc|^2-2{\Rc}^3
-\frac{\RR^3}{n-1})$, and $\EE^3=\EE_{ij}\EE_{jk}\EE_{ki}$.
\end{lemma}

Consider the special case that $\gamma=2$, i.e.,
$$f=\frac{|\EE|^2}{\RR^{2}}=\frac{|\Rc|^2}{\RR^{2}}-
\frac{1}{n},$$ we have

\begin{lemma} \label{fe} Under the Ricci flow, we have
\begin{align*}
\ppt f = &\,\triangle f+\frac{2}{\RR}\nabla f \cdot \nabla \RR
-\frac{2}{\RR^{4}} |\RR\nabla_i \RR_{jk}-\nabla_i \RR \RR_{jk}|^2
\\
&\,+\frac{2}{\RR^{3}}\left[-2
|\EE|^4-\frac{2}{n(n-1)}|\RR|^2|\EE|^2 -\frac{4}{n-2}\RR
\EE^3+2\RR \WW(\EE,\EE)\right]\\
= &\,\triangle f+\frac{2}{\RR}\nabla f \cdot \nabla \RR
-\frac{2}{\RR^{4}} |\RR\nabla_i \RR_{jk}-\nabla_i \RR \RR_{jk}|^2
\\
&\,+4{\RR}\left[- f^2-\frac{1}{n(n-1)}f
-\frac{2}{n-2}\frac{\EE^3}{\RR^3}+\frac{1}{\RR^3}
\WW(\EE,\EE)\right].
\end{align*}
\end{lemma}
To estimate the right-hand side of the above equation, we calim
that there exist positive constants $c_1$, $c_2$ depending only on
$n\geq 3$, such that
$$\left|\frac{2}{n-2}\EE^3\right|\leq c_1|\EE|^3,$$ and
$$|\WW(\EE,\EE)|\leq c_2|\WW||\EE|^2.$$

\begin{remark}
In the above estimates, $c_1 \sim \frac{2}{n(n-2)}$ and $c_2 \sim
n(n-1)(n-2)(n-3)$.
\end{remark}
Plugging these two inequalities into Lemma \ref{fe}, we derive
that

\begin{lemma} Under the Ricci flow, we have
\begin{align*}
\ppt f \leq &\,\triangle f+\frac{2}{\RR}\nabla f \cdot \nabla \RR
+4{\RR}\left[- f^2-\frac{1}{n(n-1)}f
+c_1f^{3/2}+c_2\frac{|\WW|}{\RR} f\right]\\
= &\,\triangle f+\frac{2}{\RR}\nabla f \cdot \nabla \RR
-4{\RR}f\left[ f-c_1f^{1/2}+\frac{1}{n(n-1)} -c_2\frac{|\WW|}{\RR}
\right].
\end{align*}
\end{lemma}

Combining the above inequality and using maximum principle, this
lead to the following

\begin{lemma}Under the Ricci flow, there exists $C_1=C_1(c_1, g(0))\geq c_1>0$, such that
$4f(0)\leq C_1^2$, then
$$f^{1/2}\leq \frac{1}{2}C_1+\sqrt{\frac{1}{4}C_1^2
-\left(\frac{1}{n(n-1)} - c_2\max_{M\times
[0,t]}{\frac{|\WW|}{\RR}}\right)}.$$
\end{lemma}

\begin{proof}
Let's denote the right side as $\Phi(t)$, i.e.,
$$\Phi(t)=\frac{1}{2}C_1+\sqrt{\frac{1}{4}C_1^2-\left(\frac{1}{n(n-1)} -
c_2\max_{M\times [0,t]}{\frac{|\WW|}{\RR}}\right)},$$ so $\Phi (t)$
is nondecreasing, by our choice of $C_1$,
$f^{\frac12}(0)\leq \frac{C_1}{2} \leq \Phi(0)$. And $f$ satisfies
\begin{align*}
\ppt f \leq \triangle f+\frac{2}{\RR}\nabla f \cdot \nabla \RR
-4{\RR}f\left[ f-C_1f^{1/2}+\frac{1}{n(n-1)} -c_2\frac{|\WW|}{\RR}
\right].
\end{align*} Notice that $\sqrt{f(0)}\leq \frac{C_1}{2}\leq
\Phi(0)$, then it follows from the maximum principle that
$f^{\frac12}(t)\leq \Phi(t)$. Since if $\max
f^{\frac12}(t)>\Phi(t)$, then we have
$$\ddt^{+} \max f(t) \leq 0.$$
\end{proof}

\begin{proof}(Theorem \ref{thm:einstein}) The inequality
(\ref{pr}) now follows from standard inequalities.
\end{proof}

As a direct consequence of Theorem \ref{thm:einstein}, we have the following claim:

\begin{corollary} \label{1st} Let $(M, g(t))$, $t\in [0,T)$, be an maximal
solution to the Ricci flow, here $T<\infty$. Then  we have
\begin{enumerate}
\item  either $\lim \sup_{[0,T)} \RR=\infty$,

\item or $\lim \sup_{[0,T)} \RR<\infty$ but $\lim \sup_{[0,T)}
\frac{|\WW|}{\RR}\rightarrow \infty$.
\end{enumerate}
\end{corollary}

\begin{proof}
For any finite time singularity, the whole Riemannian curvature
(or Ricci curvature) blows up at $T$. The Riemannian curvature
tensor is decomposed into the scalar curvature part $\RR$, the
traceless Ricci tensor $\EE$ and the Weyl tensor $\WW$. Following
from Theorem \ref{thm:einstein}, $\EE$ can not blow up if $\RR$
and $\frac{|\WW|}{\RR}$ are both bounded, hence the statement
follows.
\end{proof}

Another interesting application of (\ref{pr}) is the case when the
Weyl tensor $|\WW|$ is controlled by the scalar curvature $\RR$
and traceless Ricci tensor $|\EE|$.

\begin{theorem} (Positive Isotropic Curvature) Let $(M\sp n, g(t))$, $t\in [0,T)$, be an maximal
solution to the Ricci flow, here  $n\geq 4$ is a even positive
integer. Assuming that $g(0)$ has positive isotropic curvature,
 then  we have $\lim
\sup_{[0,T)} \RR=\infty$.

\end{theorem}

\begin{proof}
Since positive isotropic curvature is preserved by the Ricci flow,
moreover, it implies positive Weitzenb\"{o}ck operator in even
dimensions. So we have
$$\PP=\Rc \circ g -2\Rm=
\frac{(n-2)\RR}{n(n-1)}g\circ g+\frac{n-4}{n-2} \EE\circ g
-\WW>0,$$  and $\WW$ is traceless, this implies that
$$\frac{|\WW|}{\RR}<c_3\frac{|\EE|}{\RR}+
c_4.$$ Substituting this into (\ref{pr}), it
follows from elementary inequalities that $$\frac{|\EE|}{\RR} \leq
C.$$ Hence it follows that
$$\frac{|\WW|}{\RR} \leq C.$$ Since positive isotropic curvature implies scalar curvature $\RR>0$, the maximal existence
time $T<\infty$, so $$\lim \sup_{[0,T)} \RR=\infty.$$

\end{proof}

\begin{remark}The author is very grateful to Professor Maria Helena Noronha for several discussions about isotropic curvature and Weitzenb\"{o}ck operator.
\end{remark}

\section{\textbf{Finite-time Singularities, Dilation Limits and Singularity Models}}

In this section, we will use the pinching estimate in Section 3 to study the dilation limit of Ricci flow solutions. We first need to
introducing some notations. In \cite[Sect. 16]{Hsurvey}, Hamilton
introduced the notion of singularity model, roughly speaking, these are dilation limits of the Ricci flow. We briefly describe the strategy here, to find out exact details about how to dilate singularities based on rate of blowup of the curvature, see \cite[Section 16]{Hsurvey} or \cite[Chapter 8]{chowknopf1}. If we dilate the solution to the Ricci flow about a sequence of points and times $(x_i, t_i)$, where $x_i \in M$ and $t_i\rightarrow T$, we may choose the sequence of points and time so that $|\Rm|(x_i, t_i)$ is comparable to the global maximum over the space $M$ and sufficiently large previous time intervals. We now can define a sequence of pointed dilation solutions $(M, g_i(t), x_i)$ by:
 $$g_i(t)=|\Rm|(x_i,t_i)\cdot g(t_i+\frac{t}{|\Rm|(x_i,t_i)}),$$
 for time interval
 $$-t_i |\Rm|(x_i, t_i)\leq t <(T-t_i)|\Rm|(x_i, t_i),$$
 such that the curvature $|\Rm|_{g_i} (x_i, 0)=1$ and the maximum of the
 (Riemannian) curvature of $g_i$ becomes uniformly bounded, hence we
 have a sequence of solutions to the Ricci flow. For finite time
 singularities on closed manifolds, Perelman's No Local Collapsing
 Theorem \cite{perelman1} provides the injectivity radius estimate,
 which is necessary to obtain a noncollapsed limit. Then we can apply Hamilton's Cheeger-Gromov type compactness theorem \cite{Hcomp} to extract a limit solution of the Ricci flow. This is a complete solution to the Ricci flow with bounded curvature. If the solution is Type I, it is an ancient solution; if the solution is Type II, then it is an eternal solution. It is worth to mention that in dimension $3$, all dilation limits have nonnegative sectional curvature due to the pinching estimate of Hamilton \cite{Hsurvey} and T. Ivey \cite{Isoliton}.\\

Our main result in this section is the following:

\begin{theorem}\label{4.1} Let $(M, g(t))$, $t\in [0,T)$, be an maximal
solution to the Ricci flow with positive scalar curvature. Then  we have one of the following:
\begin{enumerate}
\item either $\lim \sup_{[0,T)} R=\infty$,

\item or if $\lim \sup_{[0,T)} R<\infty$, then $\lim \sup_{[0,T)}
\frac{|\WW|}{\RR}=\infty$. This must be a Type II maximal solution,
furthermore, the dilation limit must be a complete Ricci-flat solution
with $\max |\WW|=1$.
\end{enumerate}
\end{theorem}

\begin{remark}
In case (2) of Theorem \ref{4.1}, $\lim \sup_{[0,T)}
\frac{|\WW|}{\RR}=\infty$ is equivalent to $\lim \sup_{[0,T)}
|\WW|=\infty$.
\end{remark}

We first consider Type I solutions, which has been studied extensively recently by  Enders, Muller and Topping \cite{emt10}, Le and Sesum \cite{ls10}, also by Q. S. Zhang and the author \cite{cz10}.

\begin{corollary} If the solution of the Ricci flow is a Type I solution,
then we have $\frac{|\WW|}{\RR}$ is bounded, hence $\RR\rightarrow
\infty$.
\end{corollary}
\begin{proof}
If the solution is of Type I and $\frac{|\WW|}{\RR}$ is unbounded.
Since $\RR>0$ has a lower bound, $|\WW|$ needs to blow up at some
point $p$, and hence $|\Rm|$ also blows up at $p$. By
\cite[Theorem 1.8]{emt10}, scalar curvature blow up is equivalent
to whole curvature blow up, and all the blow up rates are same. So
$|\Rm|(p)$ or $|\WW|(p)$ is comparable to the maximum blow up
curvature, but the scalar curvature $\RR(p)$ also blows up at the same rate. Hence $\frac{|\WW|}{\RR}$ has to be bounded and $\RR\rightarrow
\infty$.
\end{proof}

\begin{remark}
This was essentially proven in \cite{emt10} and \cite{ls10},
notice that we used the fact from \cite[Theorem 1.8]{emt10}
that all Type I singularity notions are equivalent, so we do not
provide an independent proof here.
\end{remark}

\begin{remark} From  \cite{emt10}, \cite{ls10} and \cite{cz10}, such
type I dilation limit must be a nontrivial gradient Ricci
solitons. Notice that the dilation limit can not be Ricci flat, otherwise this
contradicts  a theorem \cite[Theorem 3]{prs09} of S. Pigola, M.
Rimoldi and A. G. Setti.
\end{remark}

Combine Corollary \ref{1st} and the above discussion, we now can
finish our proof of Theorem \ref{4.1}:

\begin{proof}
We assume that $\lim \sup_{[0,T)} \RR <\infty$, since this is a
finite time singularity, the whole Riemannian curvature tensor
$|\Rm|$ blows up. So $\lim \sup_{[0,T)} \frac{|\WW|}{\RR} =
\infty$, otherwise by (\ref{pr}), the traceless Ricci tensor is
also bounded and contradicts it is a finite-time singularity.
Since $R>0$ has a lower bound, we have $\lim \sup_{[0,T)} |\WW| =
\infty$. In this case,  rescaling with respect to $|\Rm|$ is same
as rescaling with respect to $|\WW|$. Hence we have
$|\tilde{\WW}|=1$ at the origin and new time $0$ after dilation.
By Theorem \ref{thm:einstein},  we have
\begin{align} \frac{|\EE|}{|\WW|_{max}} \leq C_1 \frac{\RR}{|\WW|_{max}}
+C_2 \sqrt{\frac{\RR}{|\WW|_{max}}},
\end{align}
as $t \rightarrow T$, $|\tilde{\EE}| \rightarrow 0$, $\tilde{\RR}
\rightarrow 0$, hence $\tilde{\Rc} \rightarrow 0$. Hence the
dilation limit is a complete solution to the Ricci flow with bounded curvature, furthermore, it is Ricci-flat with $\max |\tilde{\WW}|=1$.

\end{proof}

\begin{remark}
It came to our attention that there is a short proof for Theorem
\ref{4.1} using blow-up argument, without using (\ref{pr}), for
the completeness of the discussion, we include it here.
\end{remark}

\begin{proof}
If both the scalar curvature $\RR$ and the Weyl tensor $\WW$ are
uniformly bounded, then after take a blow-up limit around the
sequence of points $(x_i, t_i)$ as we described at the beginning
of this section, the dilation limit will have $\tilde{\RR}=0$ and
$|\tilde{\WW}|=0$. By the evolution equation of the scalar
curvature, it must also be Ricci-flat (for this, we only need that
$\RR$ is uniformly bounded), hence the dilation limit is actually
flat. This contradicts our choice of  base points $(x_i,
t_i)$ in the blow-up procedure. The rest follows the same way as before.
\end{proof}

{\bf Acknowledgement:} The author would like to thank Professors
Bennett Chow, John Lott, Maria Helena Noronha, Duong H. Phong,
Gang Tian and Zhou Zhang for their interest and helpful
suggestions.


\def\cprime{$'$}

\end{document}